\documentclass[10pt]{article}
\usepackage{amsmath,amsthm,amssymb,color}
\usepackage{times}
\usepackage{enumerate,graphicx}

\def\titlerunning#1{\gdef\titrun{#1}}
\makeatletter
\def\author#1{\gdef\autrun{\def\and{\unskip, }#1}\gdef\@author{#1}}
\def\address#1{{\def\and{\\\hspace*{18pt}}\renewcommand{\thefootnote}{}%
\footnote {#1}}%
\markboth{\autrun}{\titrun}}
\makeatother
\def\email#1{e-mail: #1}
\def\subjclass#1{{\renewcommand{\thefootnote}{}%
\footnote{\emph{Mathematics Subject Classification (2010):} #1}}}
\def\keywords#1{\par\medskip
\noindent\textbf{Keywords.} #1}

\newtheorem{lem}{Lemma}[section]

\newtheorem{theorem}[lem]{Theorem}
\newtheorem{proposition}[lem]{Proposition}

\newtheorem{corollary}[lem]{Corollary}

\newtheorem{problem}[lem]{Problem}

\newtheorem{remark}[lem]{Remark}

\newcommand{\Qed}{\rule{2.5mm}{3mm}}
\newcommand{\Aut}{\hbox{Aut}}

\newcommand{\la}{\langle}
\newcommand{\ra}{\rangle}

\newcommand{\ZZ}{\mathbb{Z}}

\newcommand{\F}{\mathrm{F}}
\newcommand{\A}{\mathrm{A}}
\newcommand{\B}{\mathrm{B}}

\newcommand{\Fix}{\mathrm{Fix}}

\newcounter{claim}

\renewcommand{\theclaim}{\arabic{claim}}

\newcounter{case}

\renewcommand{\thecase}{\arabic{case}}

\newcounter{subcase}

\numberwithin{subcase}{case}

\newcounter{subsubcase}

\numberwithin{subsubcase}{subcase}

\newenvironment{proofT}{{\noindent \sc Proof of Theorem~\ref{thm:main}.}}{\hfill $\Qed$ \\}

\numberwithin{equation}{section}

\begin{document}

%%%%% To ease editing, add:

\baselineskip=17pt

%%%%%%%%%%%%%%%%

%% In the running head, give an abbreviation of the title. 
\titlerunning{Odd extensions of transitive groups via symmetric graphs}

\title{Odd extensions of transitive groups\\ via symmetric graphs}

\author{Klavdija Kutnar
\and 
Dragan Maru\v si\v c}

\date{}

\maketitle

\address{ K.~Kutnar: \\
\indent \indent University of Primorska, UP IAM, Muzejski trg 2, 6000 Koper, Slovenia;\\
\indent \indent University of Primorska, UP FAMNIT, Glagolja\v ska 8, 6000 Koper, Slovenia;\\ 
\indent \indent \email{klavdija.kutnar@upr.si}
\and
D.~Maru\v si\v c (Corresponding author): \\
\indent \indent University of Primorska, UP IAM, Muzejski trg 2, 6000 Koper, Slovenia;\\
\indent \indent University of Primorska, UP FAMNIT, Glagolja\v ska 8, 6000 Koper, Slovenia;\\
\indent \indent IMFM, Jadranska 19, 1000 Ljubljana, Slovenia;\\
\indent \indent \email{dragan.marusic@upr.si}}


~\vspace*{-2.5cm}
 
 %\end{document}

\begin{abstract}
When dealing with symmetry properties of mathematical objects, one of the fundamental 
questions is to determine their full automorphism group.
In this paper this question is considered in the context of even/odd permutations
dichotomy. More precisely:  
when is it that existence of automorphisms  
acting as even permutations on the vertex set of a graph, 
called {\em even automorphisms}, forces existence of automorphisms
that  act as odd permutations, 
called {\em odd automorphisms}.
As a first step towards resolving the above question, 
a complete information on existence of odd automorphisms in cubic
symmetric graphs is given.

%% Keywords are optional
\end{abstract}

%%%%%%%%%%%%%%%%%%%%%%%%%%%%%%%%%%%%%%%%%%%%%%%%%%%%%%
%%%%%%%%%%%%%%%%%%%%%%%%%%%%%%%%%%%%%%%%%%%%%%%%%%%%%%
%%%%%%%%%%%%%%%%%%%%%%%%%%%%%%%%%%%%%%%%%%%%%%%%%%%%%%
\section{Introductory remarks}
\label{sec:intro}
\noindent
%%%%%%%%%%%%%%%%%%%%%%%%%%%%%%%%%%%%%%%%%%%%%%%%%%%%%%
%%%%%%%%%%%%%%%%%%%%%%%%%%%%%%%%%%%%%%%%%%%%%%%%%%%%%%
%%%%%%%%%%%%%%%%%%%%%%%%%%%%%%%%%%%%%%%%%%%%%%%%%%%%%%

Finding the full automorphism group is one of the fundamental objectives
when dealing with symmetry properties of mathematical objects, such as for example 
vertex-transitive graphs, see
 \cite{M1,M2,M4,M6,M9,M10}. 
Many of these objects naturally display certain inherently obvious symmetries.  It is often the case, however, that certain additional symmetries, though hidden or difficult to grasp, are present. 
The goal is to find a reason for their existence and a method for describing them. 
Along these lines the above question reads as follows: 
{\em Given a (transitive) group $H$ acting on the set of vertices of a graph,
determine whether $H$ is its full automorphism group or not. When the answer is no, find a method to describe the additional automorphisms.} In other words, decide
whether the group $H$ imbeds into a larger group $G$ preserving the structure of the graph in question. 
 
In this paper such group ``extensions'' are considered by studying 
the existence of automorphisms that act as odd permutations on the vertex set of a graph. 
Such an automorphism will be referred to as {\em odd automorphism}.
Analogously, an automorphism acting as an even permutation on the vertex set
will be referred to as an {\em even automorphism}. 

The implications are more far-reaching than one would expect from the simplicity of the concept of even and odd automorphisms alone. 
Some arguments supporting this claim are given in the paragraphs following
formulation of Problems~\ref{problem-graph} and~\ref{problem-group}. 
For example, when $H$ consists of even automorphisms only, a partial answer to the above question could be given provided the structure of the graph in question  forces existence of odd automorphisms. 
We propose the  following problem.

\begin{problem}
\label{problem-graph}
Which vertex-transitive graphs admit odd automorphisms?
\end{problem}

It is convenient to consider Problem~\ref{problem-graph}  in 
the framework of orbital (di)graphs of transitive groups. Namely, isomorphism classes of vertex-transitive (di)graphs are in a one-to-one correspondence 
with orbital (di)graphs of transitive group actions, as seen below.
A transitive group $H$ acting on a set $V$ induces an action of $H$ on $V\times V$ - the corresponding orbits are called {\em orbitals} of $H$ on $V$. 
Given a union $\mathcal{O}$ of orbitals not containing the diagonal 
orbital $D=\{(v,v)\mid v\in V\}$,
the {\em orbital (di)graph} $X(H,\mathcal{O})$ of 
$H$ on $V$ with respect to $\mathcal{O}$ has vertex set
$V$ and edge set $\mathcal{O}$. 
Clearly, $H \le \Aut(X(H,\mathcal{O}))$.
If $\mathcal{O}$ consists of a single orbital
the orbital (di)graph $X(H,\mathcal{O})$ is said to be {\em basic}.
The problem below is a somewhat restricted
reformulation  of Problem~\ref{problem-graph}.
 
\begin{problem}
\label{problem-group}
Given a transitive group $H$, possibly  consisting of even permutations only,  
is there a basic orbital (di)graph of $H$ admitting an odd automorphism?
\end{problem}

Transitive groups for which the answer to Problem~\ref{problem-group}
is positive are referred to as {\em orbital-odd}.
Not all transitive groups consisting solely of even permutations 
admit orbital-odd imbeddings.
For example, the group $\mathrm{PSL}(2,17)$ acting transitively
on the set of cosets of $S_4\le \mathrm{PSL}(2,17)$  consists of even permutations only,
and is isomorphic to the full automorphism group of every basic orbital (di)graph
associated with this action. Hence, $\mathrm{PSL}(2,17)\le S_{28}$ is not orbital-odd.
On the other hand, the alternating group $A_5$ acting
transitively on the set of cosets of
$S_3\le A_5$ (which also consists of even permutations only) is orbital-odd.
Namely, the full automorphism group of its two basic orbital graphs -
the Petersen graph and its complement -  
contains odd automorphisms (and is isomorphic to the symmetric group 
$S_5$).  
More generally, as observed in \cite{Ado}, every vertex-transitive graph of order
twice a prime   admits an odd automorphism, and hence every 
transitive group of degree twice a prime   
is orbital-odd. An essential ingredient in the proof of this result is the fact 
that the above actions of $A_5$ and $S_5$ are the only simply primitive groups of degree twice a prime -- a direct consequence of the classification of finite simple groups (CFSG).
It would be of interest to obtain a CFSG-free proof of the fact 
that every vertex-transitive graph of order twice a prime admits an odd automorphism.
Combining this together with a classical Wielandt's result on simply primitive groups of degree twice a prime being of rank $3$ 
(see \cite{W56}) -- currently the best available CFSG-free information  on such groups --
one would move a step closer to obtaining a CFSG-free proof of non-existence of simply primitive groups of degree $2p$, $p\ne 5$ a prime
(see also \cite{I1,I2,I3,S72}).

Observe that, if a graph has an odd automorphism then it must have an odd automorphism whose order is a power of $2$. In fact, an odd automorphism exists if and only if there is one in a Sylow $2$-subgroup of the automorphism group, as observed in \cite{KM}. 
Consequently, for some classes of graphs the existence of odd automorphisms is easy to determine. For instance, in Cayley graphs the corresponding regular group contains odd permutations if and only if its Sylow $2$-subgroup is cyclic. When a Sylow $2$-subgroup is not cyclic, the search for odd automorphisms has to be done outside the regular group, raising the complexity of the problem.

Studying ``extensions'' via odd automorphisms is 
also essential in obtaining new approaches for solving certain long-standing 
open problems in algebraic graph theory. 
For example, in the semiregularity problem \cite{bcc15,DM81}, which asks for
the existence of semiregular automorphisms in vertex-transitive (di)graph,
knowing that the graph admits odd automorphisms 
would be helpful in the case when the underlying transitive 
group $H$ consisting of even automorphisms is elusive (that is, without semiregular elements). One could therefore hope for 
the existence of semiregular automorphisms in the odd part of the ``extension''. 
For example, the group $\textrm{AGL}(2,9)$ acting on the set of $12$ 
lines of the affine plane $AG(2,3)$ is elusive and consists of even permutations only.
Its $2$-closure, however, contains semiregular as well as odd permutations, in fact it contains
an odd semiregular permutation of order $4$.

Further, in the hamiltonicity problem \cite{L70} as well as in the snark problem for cubic Cayley graphs \cite{A}, 
important progress was recently obtained by combining algebraic methods with the theory of 
regular maps on surfaces \cite{GM07,GKMM12,HKM14}. 
Knowing the graphs admit odd automorphisms will most likely 
be helpful in our quest to solve the remaining open cases.

Our ultimate goal is to build a theory that will allow us to decide whether a given graph, 
admitting a transitive group action, does or does not have odd automorphisms. In this 
respect the first step in our strategy consists in finding necessary and sufficient conditions 
for the existence of odd automorphisms for as large as possible class  of graphs, 
admitting a transitive action of a group, with certain suitably imposed constraints. 
It is natural to expect that the solution to the question of  
existence of odd automorphisms will have some degree of arithmetic flavor.
As may be suggested by the results of this paper this is true, but only
up to a certain extent. 
 
In approaching Problems~\ref{problem-graph} and~\ref{problem-group}
in a systematic manner it
seems natural to start with the analysis of vertex-transitive graphs
of smallest possible (non-trivial) valency $3$. Such graphs fall into three classes:
graphical regular representations of groups,
vertex-transitive graphs with two edge orbits,
and symmetric graphs.
As already mentioned, graphs in the class of (cubic) 
graphical regular representations of groups admit odd automorphisms
if and only if Sylow $2$-subgroups of their automorphism groups are cyclic.  
Cubic vertex-transitive graphs with two
edge orbits (together with the corresponding tetravalent locally imprimitive symmetric graphs) 
will be considered in a separate paper.

In this  paper we give a complete solution to Problem~\ref{problem-graph}
for cubic symmetric graphs. 
(As a consequence a partial solution to Problem~\ref{problem-group}
is obtained.)
This class of graphs   
was first studied by Foster \cite{F32} and has been the source of
motivation for various research directions in algebraic graph theory.
In Foster census  \cite{B88},
a list of such graphs of order up to $512$ was produced. 
With the availability of advanced computational tools \cite{Mag}
the list was recently expanded to
graphs of order up to $10.000$, see \cite{CD,CD02}. 
Many well-known graphs are cubic symmetric graphs, arising  in connection with 
certain open problems in graph theory (such as the hamiltonicity problem \cite{L70}), 
and more generally as an important concept in
other areas of mathematics (such as finite geometry \cite{XXX09}). 
For example, the Petersen graph $\F010\A$ and the Coxeter graph $\F028\A$
are the only known symmetric graphs of order greater than $2$ without a Hamilton cycle.
Next, all bipartite cubic symmetric graphs of girth at least $6$ are Levi graphs
of $(n_3)$ configurations. For example,
the Heawood graph $\F014\A$,  
the Moebius-Kantor graph $\F016\A$,
the Pappus graph $\F018\A$ 
and the Desargues graph $\F020\B$ are Levi graphs of $(7_3)$, $(8_3)$, $(9_3)$ and $(10_3)$
configurations, respectively. 
(Hereafter the notation $\F n\A$, $\F n \B$, etc. will refer to the corresponding graphs in
the Foster census~\cite{B88,CD02}.)
There are $17$ possible types of cubic symmetric
 graphs (see Table~\ref{tab:cubic-types} in 
Subsection~\ref{sec:cubic-sym}), and 
the existence of odd automorphisms depends on these types and orders
of graphs in question.

Results of this paper do imply that the solution to Problem~\ref{problem-graph}  
is to some degree arithmetic.  
However, some special cases sprouting out in the solution for cubic
symmetric graphs suggest that the even/odd question is likely to uncover 
certain more complex structural properties of graphs (and combinatorial objects
in general), that go beyond simple arithmetic conditions.

The following is the main result of this paper.
   
\begin{theorem}
\label{thm:main}
Let $X$ be a cubic symmetric graph of order $2n$. 
In  Table~\ref{tab:odd}  a full information on 
existence of odd automorphisms in $X$ is given. In particular, the following statements hold:
\begin{enumerate}[(i)]
\itemsep=0pt
\item If $X$ is of type 
$\{1\}$, $\{1,2^1,2^2,3\}$, $\{2^1,2^2,3\}$, $\{2^1,3\}$, 
$\{2^2,3\}$, $\{1,4^1\}$, $ \{4^2\}$, $\{1,4^1,4^2,5\}$ or $\{4^1,4^2,5\}$ 
 then it has odd 
automorphisms  if and only if $n$ is odd.

\item If $X$ is of type $\{2^1\}$, $\{3\}$, $\{4^1\}$ or $\{5\}$ 
then it has odd automorphisms if and only if $n$ is odd 
and $X$ is bipartite. % {\color{red}($\Aut(X)$ is not simple)}. 

%\item If $X$ is of type  $\{1,4^1\}$ 
%then there exist odd automorphisms in $\Aut(X)$ if and only if $n$ is odd.  

\item If $X$ is of type $\{1,2^1\}$ then it has odd automorphism if and only if either $n$ is odd, or $n=2^{k-1}(2t+1)$, $k\ge 2$ and $X$ is a $(2t+1)$-Cayley graph on a cyclic group of order $2^k$. 

\item If $X$ is of type $\{2^2\}$, $\{4^1,5\}$ or $\{4^2,5\}$ then it has no odd automorphisms.
%\item {\color{red} If $X$ is of type $\{4^2\}$ then there exist odd automorphisms in 
%$\Aut(X)$ if and only if $n$ is odd. }

%\item If $X$ is of type $\{1,2^1,2^2,3\}$, $\{2^1,2^2,3\}$, $\{2^1,3\}$, 
%$\{2^2,3\}$, $ \{4^2\}$, $\{1,4^1,4^2,5\}$ or $\{4^1,4^2,5\}$ then there exist odd 
%automorphisms in $\Aut(X)$ if and only if $n$ is odd. 
\end{enumerate}
\end{theorem}

%\vspace*{-0.6cm}
%{\footnotesize
\begin{table}[h!]
\caption{\label{tab:odd}  
Existence of odd automorphisms in a cubic symmetric graph $X$ of order $2n$.}
$$
\begin{array}{ccc}
\hline
%&&\\
\textrm{Type} & \textrm{Odd automorphisms exist if and only if} & \textrm{Comments} \\
 		%	     & \textrm{} &  \\
\hline\hline
%&&\\ 
\{1\}		& n \textrm{ odd } & \textrm{ Prop.~\ref{pro:cubic-1-regular} }\\
%&&\\
\{1,2^1\}   & n \textrm{ odd, or } n=2^{k-1}(2t+1) \textrm{ and $X$ is a (2t+1)-Cayley  }& \textrm{ Prop.~\ref{pro:cubic-2-regular-2mod4} and~\ref{pro:cubic-2-regular-0mod4}}\\
                 &  \textrm{graph on a cyclic group of order $2^k$, where $k\ge 2$}& \\
%&&\\ 
\{2^1\} & n \textrm{ odd and $X$   bipartite } & \textrm{ Prop.~\ref{pro:cubic-s-regular-without-s-1} }\\
%&&\\
\{2^2\} &    \textrm{ never } & \textrm{ Prop.~\ref{pro:cubic-2-2} }\\
%&&\\
\{1,2^1,2^2,3\} & n   \textrm{ odd } & \textrm{ Prop.~\ref{pro:3-2-odd} }\\
%&&\\
\{2^1,2^2,3\} & n   \textrm{ odd } & \textrm{ Prop.~\ref{pro:3-2-odd} }\\
%&&\\
\{2^1,3\} & n   \textrm{ odd } & \textrm{ Prop.~\ref{pro:3-2-odd} }\\
%&&\\
\{2^2,3\} & n   \textrm{ odd } & \textrm{ Prop.~\ref{pro:3-2-odd} }\\
%&&\\
\{3\} & n   \textrm{ odd and $X$   bipartite } & \textrm{ Prop.~\ref{pro:cubic-s-regular-without-s-1} }\\
%&&\\
\{1,4^1\} & n   \textrm{ odd   } & \textrm{ Prop.~\ref{pro:cubic-s-regular-without-s-1} }\\
%&&\\
\{4^1\} & n   \textrm{ odd and $X$   bipartite } & \textrm{ Prop.~\ref{pro:cubic-s-regular-without-s-1} }\\
%&&\\
\{4^2\} &  n  \textrm{ odd } & \textrm{ Prop.~\ref{pro:cubic-4-2x} }\\
%&&\\
\{1,4^1,4^2,5\} & n    \textrm{ odd } & \textrm{ Prop.~\ref{pro:5-4-oddx} }\\
%&&\\
\{4^1,4^2,5\} & n    \textrm{ odd } & \textrm{ Prop.~\ref{pro:5-4-oddx} }\\
%&&\\
\{4^1,5\} &      \textrm{ never } & \textrm{ Prop.~\ref{pro:5-4-oddx} }\\
%&&\\
\{4^2,5\} &      \textrm{ never } & \textrm{ Prop.~\ref{pro:5-4-oddx} }\\
%&&\\
\{5\} & n   \textrm{ odd and $X$   bipartite } & \textrm{ Prop.~\ref{pro:cubic-s-regular-without-s-1} }\\
%&&\\
\hline\hline
\end{array}
$$
\end{table}
%}

As a direct consequence of Theorem~\ref{thm:main}
we have the following corollary about transitive groups
with a self-paired suborbit of length $3$. 

\begin{corollary}
\label{cor:groups}
Let $H$ be a transitive group with a suborbit of length $3$
giving rise to an orbital graph satisfying conditions for the existence
of odd automorphisms from Table~\ref{tab:odd}. Then $H$ is orbital-odd.
\end{corollary}

The paper is organized as follows. 
In Section~\ref{sec:prelim} notation and terminology is introduced and
certain results on cubic symmetric graphs, essential to the strategy of the proof of Theorem~\ref{thm:main}, are gathered. In Section~\ref{sec:cells}
the concept of rigid subgraphs and rigid cells are introduced which
will prove useful in the study of the even/odd question
for cubic symmetric graphs. In Section~\ref{sec:main} the proof of
Theorem~\ref{thm:main} is given following  a series of propositions 
in which the $17$ particular types of cubic symmetric graphs are considered.
 
%%%%%%%%%%%%%%%%%%%%%%%%%%%%%%%%%%%%%%%%%%%%%%%%%%%%%%
%%%%%%%%%%%%%%%%%%%%%%%%%%%%%%%%%%%%%%%%%%%%%%%%%%%%%%
%%%%%%%%%%%%%%%%%%%%%%%%%%%%%%%%%%%%%%%%%%%%%%%%%%%%%%
\section{Preliminaries}
\label{sec:prelim}
\noindent
%%%%%%%%%%%%%%%%%%%%%%%%%%%%%%%%%%%%%%%%%%%%%%%%%%%%%%
%%%%%%%%%%%%%%%%%%%%%%%%%%%%%%%%%%%%%%%%%%%%%%%%%%%%%%
%%%%%%%%%%%%%%%%%%%%%%%%%%%%%%%%%%%%%%%%%%%%%%%%%%%%%%

%%%%%%%%%%%%%%%%%%%%%%%%%%%%%%%%%%%%%%%%%%%%%%%%%%%%%%
%%%%%%%%%%%%%%%%%%%%%%%%%%%%%%%%%%%%%%%%%%%%%%%%%%%%%%
%%%%%%%%%%%%%%%%%%%%%%%%%%%%%%%%%%%%%%%%%%%%%%%%%%%%%%
\subsection{Notation}
\noindent
%%%%%%%%%%%%%%%%%%%%%%%%%%%%%%%%%%%%%%%%%%%%%%%%%%%%%%
%%%%%%%%%%%%%%%%%%%%%%%%%%%%%%%%%%%%%%%%%%%%%%%%%%%%%%
%%%%%%%%%%%%%%%%%%%%%%%%%%%%%%%%%%%%%%%%%%%%%%%%%%%%%%

Throughout this paper graphs are finite, simple, undirected
and connected, unless specified otherwise.
Given a graph $X$ we let $V(X)$, $E(X)$, $A(X)$ and 
$\Aut (X)$ be the vertex set, the edge set, 
the arc set and the automorphism group of $X$, 
respectively. For adjacent vertices $u$ and $v$ in $X$, we  
denote the corresponding edge by $uv$. If $u\in V(X)$ then  $N(u)$ denotes the set of 
neighbors of $u$ and $N^i(u)$ denotes the set of vertices at distance  $i>1$  from $u$.
A graph $X$ is said to be {\em cubic} if $|N(u)|=3$ for every vertex $u\in V(X)$.
A sequence $(u_0,u_1,u_2,\ldots, u_s)$ of distinct vertices in a graph is called an 
{\em $s$-arc} if $u_i$ is adjacent to $u_{i+1}$ for every $i\in \{0,1,\ldots,s-1\}$.
For $S\subseteq V(X)$ we let $X[S]$ denote  the induced subgraph of $X$ on  $S$.
For a partition $\cal W$ of $V(X)$, we let $X_{\cal W}$ 
be the associated {\em quotient graph} of
$X$ {\em relative to} $\cal W$, that is, the graph 
with vertex set $\cal W$ and edge set induced 
by the edge set $E(X)$ in a natural way. When $\cal W$ consists of orbits of a subgroup 
$H$ of $\Aut(X)$ we denote  $X_{\cal W}$ by $X_H$, and by
$X_a$ when $H=\la a\ra$ is a cyclic group generated by an automorphism $a$.

%By an $n$-cycle we shall always mean a cycle with $n$ vertices. 
%We will use the symbol $\ZZ_n$ for both, the cyclic group of order $n$
%and the ring of integers modulo $n$.  

A subgroup $G \leq \Aut(X)$ is said to be {\em vertex-transitive}, {\em edge-transitive}
and {\em arc-transitive} provided it acts transitively on the sets  
of vertices, edges and arcs 
of $X$, respectively. A graph is said to be {\em vertex-transitive}, {\em edge-transitive},
and {\em arc-transitive} if its automorphism group is vertex-transitive, 
edge-transitive and 
arc-transitive, respectively. An arc-transitive graph is also called {\em symmetric}.
A subgroup $G\leq \Aut X$ is said to be {\em $s$-regular} if it acts transitively 
on the set of $s$-arcs and the stabilizer of an  $s$-arc in $G$ is trivial. 

A {\em quasi-dihedral group} $QD_{2^n}$ (sometimes 
also called a {\em semi-dihedral group})
is a non-abelian group of order $2^n$ with a presentation
$
\la r,s\mid r^{2^{n-1}}=s^2=1, srs=r^{2^{n-2}-1}\ra.
$

To end  this subsection let us recall that if
a transitive permutation group $G$ on a set $V$
contains an odd permutation then the intersection $G\cap \textrm{Alt}(V)$
of $G$ with the alternating group ${Alt}(V)$ on $V$
is its index $2$ subgroup.
This will be used in several places in subsequent sections.
 
%%%%%%%%%%%%%%%%%%%%%%%%%%%%%%%%%%%%%%%%%%%%%%%%%%%%%%
%%%%%%%%%%%%%%%%%%%%%%%%%%%%%%%%%%%%%%%%%%%%%%%%%%%%%%
%%%%%%%%%%%%%%%%%%%%%%%%%%%%%%%%%%%%%%%%%%%%%%%%%%%%%%
\subsection{Cubic symmetric graphs}
\label{sec:cubic-sym}
\indent
%%%%%%%%%%%%%%%%%%%%%%%%%%%%%%%%%%%%%%%%%%%%%%%%%%%%%%
%%%%%%%%%%%%%%%%%%%%%%%%%%%%%%%%%%%%%%%%%%%%%%%%%%%%%%
%%%%%%%%%%%%%%%%%%%%%%%%%%%%%%%%%%%%%%%%%%%%%%%%%%%%%%

In \cite{T47} Tutte proved that every finite cubic symmetric  graph is $s$-regular 
for some $s \le 5$. A further deeper insight into the structure cubic symmetric graphs is due to Djokovi\'c and Miller \cite{DM80} who proved 
that a vertex  stabilizer in an $s$-regular  subgroup of automorphisms of a  
cubic symmetric graph is isomorphic to $\ZZ_3$, $S_3$, $S_3\times\ZZ_2$, $S_4$, 
or $S_4\times \ZZ_2$ depending on whether $s=1,2,3,4$ or $5$, respectively. Consequently, the automorphism group
of a cubic symmetric graph of order $n$ is of order $3\cdot 2^{s-1}  n$.
Djokovi\'c and Miller \cite{DM80} also proved that for $s\in\{1,3,5\}$ there is just one possibility for edge stabilizers, while there exists two possibilities for $s\in\{2,4\}$,
see Table~\ref{tab:vertex-and-edge-stabilizers}. 
In particular, for $s=2$ the edge stabilizer is either isomorphic to $\ZZ_2\times\ZZ_2$
or $\ZZ_4$, and for $s=4$ the edge stabilizer is either isomorphic to the dihedral group 
 $D_{16}$ of order $16$ or to the quasi-dihedral group $QD_{16}$ of order $16$.

\begin{table}[h!]
\caption{\label{tab:vertex-and-edge-stabilizers}  The list of all possible pairs of vertex and edge stabilizers in cubic $s$-regular graphs. }
$$
\begin{array}{ccc}
\hline 
%&&\\
s &   \Aut(X)_v & \Aut(X)_{e} \\%[4pt]
\hline	\hline 
%&&\\ 
1 & \ZZ_3 & id \\%[4pt]
2 & S_3 & \ZZ_2^2 \textrm{ or } \ZZ_4 \\%[4pt]
3 & S_3 \times \ZZ_2 & D_8 \\%[4pt]
4 & S_4& D_{16} \textrm{ or } QD_{16} \\%[4pt]
5 & S_4 \times \ZZ_2 & (D_8\times\ZZ_2)\rtimes\ZZ_2 \\%[4pt]
\hline 
\end{array}
$$
\end{table}

The automorphism group of any finite symmetric cubic graph is an epimorphic image of one of the following seven groups:{\small
\begin{eqnarray} 
\label{G1}
G_1&=&\la h,a \mid h^3=a^2=1\ra,\\%[8pt]
\label{G21}
G_2^1&=&\la h,a, p \mid h^3=a^2=p^2=1, apa=p, php=h^{-1}\ra,\\%[8pt]
\label{G22}
G_2^2&=&\la h,a, p \mid h^3=p^2=1, a^2=p, php=h^{-1}\ra,\\%[8pt]
\label{G3}
G_3&=&\la h,a, p,q \mid h^3=a^2=p^2=q^2=1, apa=q, qp=pq, 
ph=hp, qhq=h^{-1}\ra,\\%[8pt]
\label{G41}
G_4^1&=&\la h,a, p,q,r \mid h^3=a^2=p^2=q^2=r^2=1, apa=p, aqa=r, h^{-1}ph=q, \\%[8pt]
\nonumber
&&  h^{-1}qh=pq, rhr=h^{-1}, pq=qp, pr=rp, rq=pqr\ra,\\%[8pt]
\label{G42}
G_4^2&=&\la h,a, p,q,r \mid h^3=p^2=q^2=r^2=1, a^2=p, a^{-1}qa=r, h^{-1}ph=q,    \\%[8pt]
\nonumber
&&  h^{-1}qh=pq, rhr=h^{-1}, pq=qp, pr=rp, rq=pqr\ra,\\%[8pt]
\label{G5}
G_5&=&\la h,a, p,q,r,s \mid h^3=a^2=p^2=q^2=r^2=s^2=1,  apa=q, ara=s,    \\%[8pt]
\nonumber
&& h^{-1}ph=p,  h^{-1}qh=r, h^{-1}rh=pqr, shs=h^{-1},  pq=qp, pr=rp, ps=sp,  \\%[8pt]
\nonumber 
&&qr=rq, qs=sq, sr=pqrs\ra.
\end{eqnarray}
}
This implies that an arc-transitive subgroup  of a cubic symmetric graph   is a quotient of one of these seven
groups by some normal torsion-free subgroup. In particular, an $s$-regular subgroup of automorphisms of 
a cubic symmetric graph is a quotient group of a group $G_s$, if $s\in\{1,3,5\}$, and of a group $G_s^i$, $i\in\{1,2\}$,
if $s\in \{2,4\}$.  
Moreover, in \cite{CN09}  a complete characterization of 
admissible types of cubic symmetric graphs according to the structure of arc-transitive subgroups is given.
For example, a cubic symmetric graph $X$ is said to be of type $\{1,2^1,2^2,3\}$ if its automorphism group is $3$-regular,
admitting two $2$-regular subgroups, of which one is a quotient of the group $G_2^1$ and one of $G_2^2$, and admitting also a $1$-regular subgroup.
All possible types are summarized in 
Table~\ref{tab:cubic-types} (for details see \cite{CN09}). 

\begin{table}[h!]
\caption{\label{tab:cubic-types}   All possible types of cubic symmetric graphs.}
$$
\begin{array}{ccc|ccc|ccc}
\hline
s&\textrm{Type} &\textrm{Bipartite?}& s&\textrm{Type} &\textrm{Bipartite?}& s&\textrm{Type}&\textrm{Bipartite?}\\
\hline \hline
1& \{1\} 	&\textrm{Sometimes}	&			3 & \{2^1,3\}	&\textrm{Never}	& 		5 & \{1,4^1,4^2,5\} &\textrm{Always}\\
2 & \{1,2^1\} &\textrm{Sometimes}	&			3 & \{2^2,3\} 	&\textrm{Never}	& 		5 & \{4^1,4^2,5\} &\textrm{Always}\\
2 & \{2^1\}		&\textrm{Sometimes}&			3 & \{3\} 		&\textrm{Sometimes}	& 		5 & \{4^1,5\} &\textrm{Never}\\
2 & \{2^2\}	&\textrm{Sometimes}	&			4 & \{1,4^1\}	&\textrm{Always}	& 		5 & \{4^2,5\}&\textrm{Never}\\
3 & \{1,2^1,2^2,3\}	&\textrm{Always}	& 	4 & \{4^1\} 		&\textrm{Sometimes}	& 		5 & \{5\} &\textrm{Sometimes}\\
3 & \{2^1,2^2,3\}	&\textrm{Always}	& 	4 & \{4^2\}		&\textrm{Sometimes}	& 			& &\\		
\hline
\end{array}
$$
\end{table}

Finally, in Section~\ref{sec:main} the following result about the girth of cubic
symmetric graphs of types $\{2^2\}$ and $\{4^2\}$, extracted from 
\cite[Theorems~2.1 -- 2.2]{CN07} will be needed.
(The {\em girth} of a graph is the length of a shortest cycle contained in the graph.)

\begin{proposition}
\label{pro:girth}
A cubic symmetric graph, which is either of type $\{2^2\}$ or $\{4^2\}$, has
 girth greater than $9$.
\end{proposition}
%%%%%%%%%%%%%%%%%%%%%%%%%%%%%%%%%%%%%%%%%%%%%%%%%%%%%%
%%%%%%%%%%%%%%%%%%%%%%%%%%%%%%%%%%%%%%%%%%%%%%%%%%%%%%
%%%%%%%%%%%%%%%%%%%%%%%%%%%%%%%%%%%%%%%%%%%%%%%%%%%%%%
\section{Rigid cells}
\label{sec:cells}
\noindent
%%%%%%%%%%%%%%%%%%%%%%%%%%%%%%%%%%%%%%%%%%%%%%%%%%%%%%
%%%%%%%%%%%%%%%%%%%%%%%%%%%%%%%%%%%%%%%%%%%%%%%%%%%%%%
%%%%%%%%%%%%%%%%%%%%%%%%%%%%%%%%%%%%%%%%%%%%%%%%%%%%%%

Given a graph $X$ and an automorphism $\alpha$ of $X$
let $\textrm{Fix}(\alpha)$ denote the set of all vertices of $X$ fixed by $\alpha$.
With the assumption that $\textrm{Fix}(\alpha)\ne \emptyset$ we call
the subgraph $X[\textrm{Fix}(\alpha)]$ 
induced on $\textrm{Fix}(\alpha)$ the {\em rigid subgraph of } $\alpha$
or, in short, the $\alpha$-{\em rigid subgraph}. Every component of 
$X[\textrm{Fix}(\alpha)]$ is referred to as an {\em $\alpha$-rigid cell}.

This concept will prove useful in the study of the ``even/odd question''
for cubic symmetric graphs.
For this purpose let us fix some notation and terminology.
We will use the terms $I$-{\em tree}, $H$-{\em tree}, $Y$-{\em tree}, $A$-{\em tree}
and $B$-{\em tree} for the graphs given in Figure~\ref{fig:trees} below.
More precisely, we will denote these graphs by $I(u,v)$, $H(u,v)$, $Y(v)$,
$A(v)$ and $B(v)$, respectively. 

Let $X$ be a cubic graph of girth greater than or equal to $5$. 
For a vertex $v$ in $V(X)$ let $N(v)=\{v_1,v_2,v_3\}$ be the set of neighbors of $v$, 
and let $N^2(v)=\{v_{11}, v_{12}\} \cup \{v_{21},v_{22}\} \cup \{v_{31}, v_{32}\}$ be the second neighborhood of $v$, where $N(v_i)=\{v,v_{i1},v_{i2}\}$, $i\in \{1,2,3\}$. 
Similarly, for an edge $uv$ of $X$
let $v_1$ and $v_2$ be the remaining two neighbors of $v$, and 
let $u_1$ and $u_2$ be the remaining two neighbors of $u$.
For a vertex $v$ of $X$ we therefore have $Y(v)=X[\{v\}\cup N(v)]$,  
$A(v)=X[\{v\}\cup N(v) \cup N^2(v)]$ and 
$B(v)_i=X[\{v\}\cup N(v) \cup N^2(v)\setminus \{v_{i1},v_{i2}\}]$, where $i\in \{1,2,3\}$.
For short, we will use the notation $B(v)$ for any of the trees $B(v)_1$,
$B(v)_2$ and $B(v)_3$ (and call it a $B$-tree).
Similarly, for two adjacent vertices $u$ and $v$  of $X$ we have
$I(u,v)=X[\{u,v\}]$ and $H(u,v)=X[N(u)\cup N(v)]$.  
 Note that $H(u,v)=Y(u)\cup Y(v)$. 

\begin{figure}[h!]
\begin{center}
\includegraphics[width=90mm]{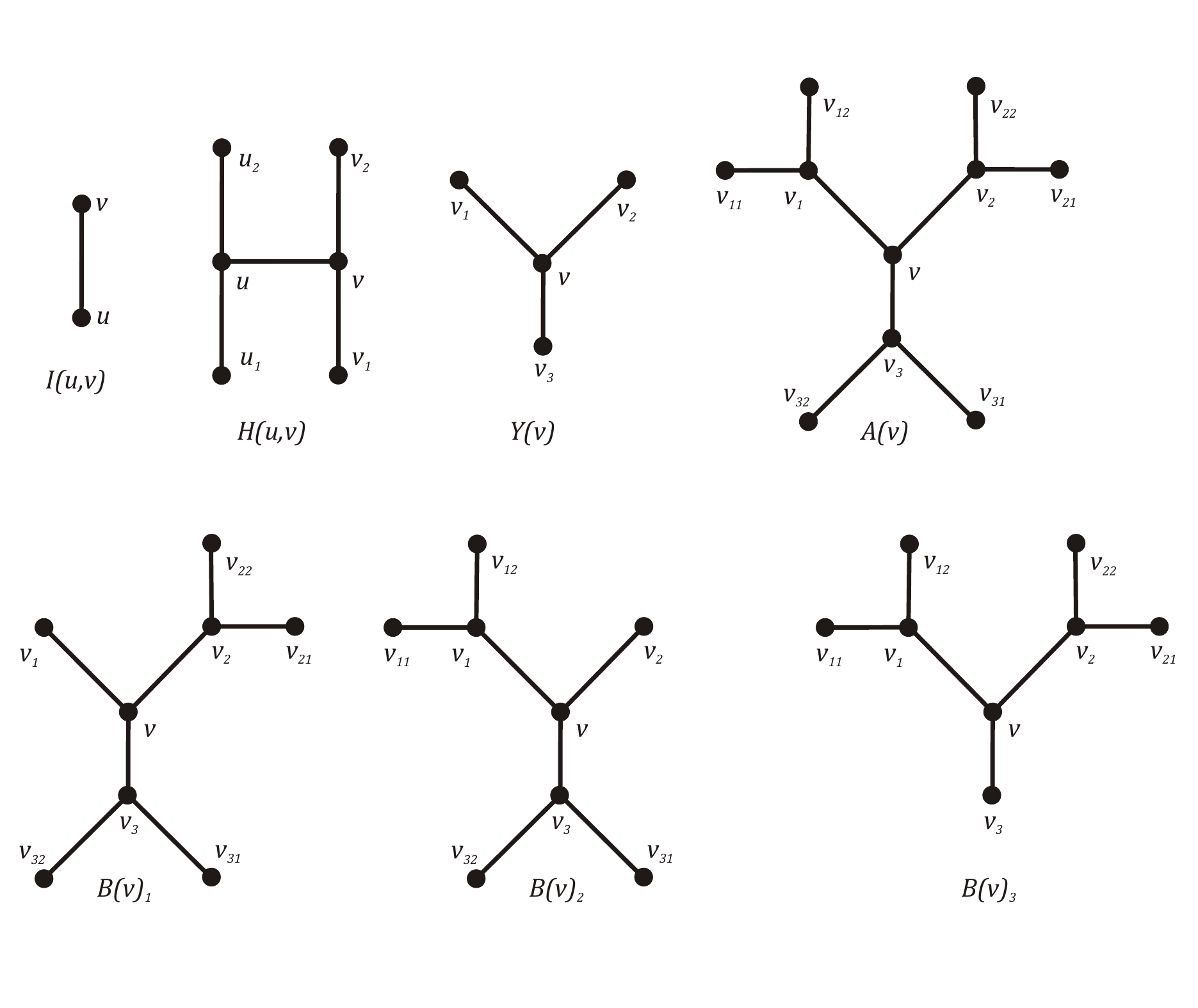}
\caption{\label{fig:trees} {\footnotesize $I$-tree, $H$-tree, $Y$-tree, $A$-tree
and three $B$-trees. }}
\end{center}
\end{figure}

We will now restrict ourselves to cubic symmetric graphs.
The structure of vertex stabilizers in cubic symmetric graphs
implies that only automorphisms of order $2$, $3$,  $4$ and $6$
can fix a vertex (see Section~\ref{sec:cubic-sym}). 
In the propositions below all possible rigid cells for such 
automorphisms are listed. 
(The exclusion of $Y$-trees as rigid cells for automorphisms
of order $4$ is proved in Proposition~\ref{pro:cell3} 
and uses Proposition~\ref{pro:cell25}.)
For convenience we let $\mathcal{I}(X)\subseteq \Aut(X)$ 
denote the set of all involutions of a cubic symmetric graph $X$
which fix some vertex of $X$, and we let
$\mathcal{S}(X)\subseteq \Aut(X)$ denote the set of all semiregular involutions of $X$. 

%{\color{red} ADD ELEMENTS OF ORDER 6. THEY CAN ALSO ME IN VERTEX STABILIZERS!!!}

\begin{proposition}
\label{pro:cell1}
Let $X$ be a cubic symmetric graph and let $\alpha\in\Aut(X)$ be an automorphism of $X$
fixing a vertex. 
\begin{enumerate}[(i)]
\itemsep=0pt
\item If $\alpha$ is of order $3$ or $6$ then the only $\alpha$-rigid cells are isolated vertices.
\item If $\alpha$ is of order $4$ then the only possible $\alpha$-rigid cells are 
$I$-trees and $Y$-trees.
\end{enumerate}
\end{proposition}

\begin{proof}
The proof of part (i) is straightforward and is omitted.
As for part (ii), if an $H$-tree, a $B$-tree or an $A$-tree was
a rigid cell of an automorphism of order $4$ then the square of
this automorphism (a non-identity element) would fix a 
path of length greater than $4$.
\end{proof}

%an automorphism of $X$ of
%order $3$ fixing a vertex

Given a cubic $s$-regular graph $X$, $s\ge 2$,
we call an involution $\sigma\in\Aut(X)$ {\em canonical
with respect to a pair of adjacent vertices} $u$ and $v$ of $X$ 
if it fixes the tree $I(u,v)$ for $s=2$ and if it fixes the tree $H(u,v)$ for $s=4$.
Similarly, we call it {\em canonical with respect to a vertex} $v$ of $X$ if
it fixes  the tree $Y(u)$ for $s=3$ and if it fixes the tree $A(v)$ for $s=5$.

The next four propositions give a full information on rigid cells of  involutions.
The proof of the first of these propositions is straightforward and is omitted.

\begin{proposition}
\label{pro:cell212}
Let $X$ be a cubic $s$-regular graph, $s\in\{1,2\}$ and let  
$\alpha\in\Aut(X)$ be an involution.
Then the following statements hold: 
\begin{enumerate}[(i)]
\itemsep=0pt
\item If $s=1$ then $\Fix(\alpha)=\emptyset$, that is, $\alpha$ is semiregular.
\item If $s=2$ then the only possible $\alpha$-rigid cells are $I$-trees.
%\item If $s=3$ then the possible $\alpha$-rigid cells are $I$-trees and $Y$-trees, with both types of cells possibly occurring simultaneously only when the graph $X$ is of type $\{3\}$.
%\item If $s=4$ then the only possible $\alpha$-rigid cells are $H$-trees. 
%\item If $s=5$ then the only possible $\alpha$-rigid 
%cells are $I$-trees, $H$-trees and $A$-trees. 
%\item If $s=5$ then the only possible $\alpha$-rigid cells are $H$-trees and $A$-trees,
%with both types of cells possibly occurring simultaneously only when the graph $X$ is of type $\{5\}$.
\end{enumerate}
\end{proposition}

\begin{proposition}
\label{pro:cell23}
Let $X$ be a cubic $3$-regular graph   and let 
$\alpha\in \mathcal{I}(X)$ be a non-semiregular involution of $X$.
Then the possible $\alpha$-rigid cells are $I$-trees and $Y$-trees,
with both types of cells possibly occurring simultaneously only when the graph $X$
is of type $\{3\}$.
\end{proposition}

\begin{proof} 
Assume that $X$ is not of type $\{3\}$.
Then $\Aut(X)$ contains a  $2$-regular  subgroup $K\le \Aut(X)$ (of index $2$).
Because of $3$-regularity of $X$ given any vertex $v\in V(X)$ there
exists a canonical involution $\sigma_v$ fixing the tree $Y(v)$ point-wise.
We now show that the only rigid cell of canonical involutions are $Y$-trees.

Suppose, on the contrary, that there exists a canonical involution $\alpha\in \Aut(X)\setminus K$ 
with both the $I$-tree and the $Y$-tree as rigid cells 
of $X[\textrm{Fix}(\alpha)]$. Then there exist a vertex $v\in V(X)$ and adjacent 
vertices $u,w\in V(X)$ such that  $Y(v)$ and  $I(u,w)$ are rigid cells of $\alpha$. 
Clearly, there also exists  an involution $\beta\in K$ whose action on 
$Y(u)\cup Y(w)=H(u,w)$ coincides with the action of $\alpha$.
It follows that the product $\alpha\beta$ acts trivially on $H(u,w)$, 
thus fixing a path of length $3$. But $X$ is $3$-regular, and so $\alpha\beta=1$. Consequently, $\alpha=\beta\in K$, a contradiction. 
\end{proof}

\begin{proposition}
\label{pro:cell24}
Let $X$ be a cubic $4$-regular graph   and let 
$\alpha\in \mathcal{I}(X)$ be a non-semiregular involution of $X$.
Then the only possible $\alpha$-rigid cells are $H$-trees.
\end{proposition}

\begin{proof} 
Note that theoretically the possible rigid cells are
the $I$-tree, the $Y$-tree and the $H$-tree. We now show
that the first two possibilities cannot occur.
Let $u$ and $v$ be adjacent vertices of $X$.
Because of $4$-regularity of $X$ there exists
an involution  $\alpha'\in \Aut(X)$ with $H(u,v)$ as its rigid cell.

Suppose that there exists an involution $\beta$ with $Y$-tree as a rigid cell.
Since vertex stabilizers are conjugate subgroups there exists a 
conjugate $\beta'$ of $\beta$ with $Y(v)$ as a rigid cell.
Obviously, $\beta'$ is different from $\alpha'$ since they have different 
rigid cells containing $v$.
But since $\alpha'$ and $\beta'$
both interchange the remaining two neighbors $v_{11}$ and $v_{12}$
of $v_1$ and the remaining two neighbors $v_{21}$ and $v_{22}$
of $v_2$ it follows that $\alpha'\beta'$ fixes point-wise the $B$-tree $B(v)$,
and thus a path of length $4$. 
Since $s=4$ it follows that 
$\alpha'=\beta'$, a contradiction. Therefore no involution can have
the $Y$-tree as a rigid cell.  

Suppose now that there exists an involution $\gamma$ with $I$-tree as a rigid cell.
Then a conjugate $\gamma'$ of $\gamma$ has $I(v,v_1)$ as a rigid cell.
Obviously, $\gamma'$ is different from $\alpha'$ since they have different 
rigid cells containing $v$ and $v_1$.
Observe that $\alpha'$ and $\gamma'$ both interchange the remaining neighbors of 
$v_{11}$ and $v_{12}$. The actions of $\alpha'$ and $\gamma'$
on these remaining neighbors of $v_{11}$ and $v_{12}$ are either different or identical.
In the first case, $Y(v_1)$ is a rigid cell of $\alpha'\gamma'$.
By previous paragraph $(\alpha'\gamma')^2\ne 1$, and moreover it fixes
a path of length greater than or equal to $4$, which is not possible in a 
$4$-regular graph.
This completes the proof of Proposition~\ref{pro:cell24}.
\end{proof}
 
\begin{proposition}
\label{pro:cell25}
Let $X$ be a cubic $5$-regular graph   and let 
$\alpha\in \mathcal{I}(X)$ be a non-semiregular involution of $X$.
Then the only possible $\alpha$-rigid cells are $H$-trees and $A$-trees, 
with both types of cells
possibly occurring simultaneously only when the graph $X$ is of type $\{5\}$.
\end{proposition}

\begin{proof} 
The stabilizer $\Aut(X)_v$
of a vertex $v\in V(X)$ is isomorphic to $S_4\times \ZZ_2$.
Hence the center of $\Aut(X)_v$ is isomorphic to $\ZZ_2$,
and moreover there are precisely $19$ involutions in $\Aut(X)_v$.
Because of $5$-regularity of $X$ there exists an involution $\alpha'$
in $\Aut(X)_v$ with either $A(v)$ or $B(v)$ as a rigid cell.

In what follows we essentially translate (into our language) the arguments 
from \cite[Lemma~1]{D79}  where it is shown that the central
involution in $\Aut(X)_v$ has $A(v)$ as a rigid cell.
Suppose that $\alpha'$ has a $B$-tree $B(v)$ as a rigid cell. Then for 
conjugacy reasons    there exist   involutions
$\alpha_1$, $\alpha_2$ and $\alpha_3$ in $\Aut(X)_v$
whose respective restrictions
to $A(v)$  are:
$(v_{11}\,v_{12})$, $(v_{21}\,v_{22})$ and 
$(v_{31}\,v_{32})$. 
Their 
products $\alpha_1\alpha_2$, $\alpha_1\alpha_3$, $\alpha_2\alpha_3$
and $\alpha_1\alpha_2 \alpha_3$ are also involutions in $\Aut(X)_v$, and so
the stabilizer of the  $2$-arc $(v_{i},v,v_{j})$, where $i,j\in\{1,2,3\}$ and $i\ne j$,
is an elementary abelian group $\ZZ_2^3$. Since $X$ is $5$-regular,
stabilizers of $2$-arcs in $X$ are conjugate subgroups, and so 
the stabilizer of any $2$-arc in $X$ is isomorphic to $\ZZ_2^3$.
Let $\beta \in \Aut(X)_{v} \cap \Aut(X)_{v_{11}}$ be such that $\beta(v_2)=v_3$
and $\beta(v_{22})=v_{31}$. (Because of $5$-regularity such an element clearly exists.)
Then $\beta\alpha_2(v_{21})=\beta(v_{22})=v_{31}$ and 
$\alpha_2\beta(v_{21})=\alpha_2(v_{32})=v_{32}$, implying that
$\alpha_2\beta\ne\beta\alpha_2$. But this is a contradiction since 
$\alpha_2$ and $\beta$ both belong to the stabilizer of
the $2$-arc $(v,v_{1},v_{11})$, which is an abelian group.
This shows that the $B$-tree cannot occur as a rigid cell in $X$. 
Moreover, $\alpha'$ has the whole of $A(v)$ as a rigid cell.
In other words, $\alpha'$ is the canonical involution for $v$.
This means that $\alpha'$ is in the kernel of the restriction of $\Aut(X)_v$
to $A(v)$. Let us also remark that the restrictions of involutions  in $\Aut(X)_v$
to $N^2(v)$ are $(v_{11}\,v_{12})(v_{21}\,v_{22})$,  
$(v_{11}\,v_{12})(v_{31}\,v_{32})$ and $(v_{21}\,v_{22})(v_{31}\,v_{32})$.

Next, we show that the $Y$-tree cannot be a rigid cell.
Assume on the contrary that there exists an involution $\sigma$
with the $Y$-tree as a rigid cell.
Then there is an involution $\sigma'$ (a conjugate of $\sigma$) 
whose rigid cell is $Y(v_1)$. 
Obviously, $\sigma'$ is different from $\alpha'$ since they have different 
rigid cells containing $v$.
But then $\alpha'\sigma'$ either fixes a path of length $5$
or it has $B(v_1)$ as a rigid cell.
If the first possibility occurs then $\alpha'=\sigma'$, 
contradicting our assumption on $\sigma'$. 
If the second possibility occurs then 
a non-identity element $\alpha'\sigma'$ has the $B$-tree
as a rigid cell. But this is not possible by Proposition~\ref{pro:cell1}
if $\alpha'\sigma'$ is of order $4$, and by the argument in the previous paragraph
if $\alpha'\sigma'$ is an involution.  
 
We now show that the $I$-tree cannot be a rigid cell.
Assume on the contrary that there exists an involution $\gamma$
with the $I$-tree as a rigid cell.
Then there is an involution $\gamma'$ (a conjugate of $\gamma$) 
whose rigid cell is $I(v_1, v_{11})$. 
Obviously, $\gamma'$ is different from $\alpha'$ since they have different 
rigid cells containing $v_1$.
But then it can be shown that 
$\alpha'\gamma'$ either has $Y(v_{11})$ or $B(v_{11})$ as a rigid cell 
or it fixes a path of length greater than $4$.
Each of these cases leads to a contradiction
by the previous arguments with the exception of the case
where $\alpha'\gamma'$ is of order $4$ and has $Y(v_{11})$ as a rigid cell.
There exists a conjugate $\delta$ of $\alpha'\gamma'$  with $Y(v_1)$ as a rigid cell.
But then $\delta\alpha'$ either has $B(v_1)$ as a rigid cell or it fixes a path
of length greater than $4$. Both possibilities clearly lead to a contradiction.
This shows that the only possible rigid cells in $X$ are $H$-trees and $A$-trees.

To finish the proof assume that $X$ is not of type $\{5\}$.
Then $\Aut(X)$ contains a $4$-regular  subgroup $K$ (of index $2$).
Because of $5$-regularity of $X$ 
it is clear that $A$-trees occur as rigid cells of non-semiregular involutions 
in $\Aut(X)\setminus K$. In order to prove that all rigid cells  of such involutions are
$A$-trees suppose, on the contrary, that an involution $\alpha\in \Aut(X)\setminus K$ 
has both the $H$-tree and the $A$-tree as rigid cells 
of $X[\textrm{Fix}(\alpha)]$. Then there exist a vertex $v\in V(X)$ and adjacent 
vertices $u,w\in V(X)$ such that  $A(v)$ and  $H(u,w)$ are rigid cells of $\alpha$. 
Clearly, there also exists  an involution $\beta\in K$ whose action on 
$H(u,w)$ coincides with the action of $\alpha$.
It follows that the product $\alpha\beta$ fixes a path of length greater than $4$. 
But then $\alpha\beta=1$, and so $\alpha=\beta\in K$, a contradiction. 
This completes the proof of  
Proposition~\ref{pro:cell25}.
\end{proof}

As an almost immediate consequence of Proposition~\ref{pro:cell25}
we can now determine rigid cells of automorphisms of order $4$.

\begin{proposition}
\label{pro:cell3}
Let $X$ be a cubic $s$-regular graph and $\alpha$ an automorphism of $X$ of
order $4$ fixing a vertex. Then the only possible $\alpha$-rigid cells 
are the $I$-trees.  
\end{proposition}

\begin{proof}
The assumption on the existence of an automorphism $\alpha$
of order $4$ fixing a vertex implies 
that either $s=4$ or $s=5$ (see Table~\ref{tab:vertex-and-edge-stabilizers}). 
By Proposition~\ref{pro:cell1} the only possible candidates for $\alpha$-rigid cells
are $I$-trees and $Y$-trees. 

Suppose that $\alpha$ has a rigid cell isomorphic to 
a $Y$-tree. Then $\alpha^2$ has a rigid cell isomorphic to an $A$-tree,
and thus $X$ is $5$-regular.
Let $v\in V(X)$ with neighbors $v_1$, $v_2$ and $v_3$. 
By assumption there exists a conjugate $\alpha'$ of $\alpha$
such that $Y(v_1)$ is an $\alpha'$-rigid cell.
Let $\beta\in\Aut(X)_v$ be the canonical involution of $v$. 
It follows that the $B$-tree 
$B(v_1)=X[\{v_1\}\cup N(v_1)\cup N^2(v_1)\setminus \{v_2,v_3\}]$ 
is contained in a rigid cell of a non-identity automorphism $\beta\alpha'\in \Aut(X)_v$.
Clearly this rigid cell does not contain $v_2$ and $v_3$, and so
it is either a $B$-tree or a rigid cell containing a path of length greater than $4$.
Both possibilities lead to a contradiction. 
We conclude that the only possible $\alpha$-rigid cells 
are $I$-trees.
\end{proof}

%%%%%%%%%%%%%%%%%%%%%%%%%%%%%%%%%%%%%%%%%%%%%%%%%%%%%%
%%%%%%%%%%%%%%%%%%%%%%%%%%%%%%%%%%%%%%%%%%%%%%%%%%%%%%
%%%%%%%%%%%%%%%%%%%%%%%%%%%%%%%%%%%%%%%%%%%%%%%%%%%%%%
\section{Proof of Theorem~\ref{thm:main}}
\label{sec:main}
\noindent
%%%%%%%%%%%%%%%%%%%%%%%%%%%%%%%%%%%%%%%%%%%%%%%%%%%%%%
%%%%%%%%%%%%%%%%%%%%%%%%%%%%%%%%%%%%%%%%%%%%%%%%%%%%%%
%%%%%%%%%%%%%%%%%%%%%%%%%%%%%%%%%%%%%%%%%%%%%%%%%%%%%%
 
%Solutions for types $\{2^1,3\}$, $\{2^2,3\}$, $\{2^1, 2^2,3\}$, $\{1, 2^1, 2^2,3\}$,
%$\{4^1, 5\}$, $\{4^2, 5\}$, $\{4^1, 4^2,5\}$, $\{1,4^1, 4^2,5\}$

The proof of Theorem~\ref{thm:main} is carried out through a series 
of propositions (Propositions~\ref{pro:cubic-1-regular} --\ref{pro:5-4-oddx}),
each dealing with particular types of cubic symmetric graphs. 
The first of these propositions gives the  proof of Theorem~\ref{thm:main} 
for cubic symmetric graphs of type $\{1\}$.
In fact it is a slightly more general for it gives a necessary and sufficient 
conditions  on existence of odd automorphisms in  
any cubic symmetric graph admitting a $1$-regular subgroup
in its automorphism group.

\begin{proposition}
\label{pro:cubic-1-regular}
Let $X$ be a cubic symmetric graph of order $2n$ 
admitting a $1$-regular subgroup $G\le \Aut(X)$.
Then there exists an odd automorphism of $X$ in $G$  if and only if $n$ is odd. 
\end{proposition}

\begin{proof}
Let $v\in V(X)$. By (\ref{G1}) the group $G$ is 
generated by $G_v=\la h\ra\cong \ZZ_3$, $v\in V(X)$, and an involution
$a$ interchanging two adjacent vertices. Automorphisms in the vertex stabilizer 
$G_v\cong \ZZ_3$ are all even automorphisms as they are of odd order.
Hence we can conclude that $G$ 
admits odd automorphisms if and only if $a$ is 
an odd automorphism. Since $a$ is semiregular the result follows.
\end{proof}

The next three propositions give the answer regarding 
existence of odd automorphisms
in cubic $s$-regular graphs, $s\ge 2$, having no $(s-1)$-regular subgroup of 
automorphisms. The first proposition in the series 
shows that all automorphisms of cubic symmetric graphs of type $\{2^2\}$ are even.
 
\begin{proposition}
\label{pro:cubic-2-2}  
Let $X$ be a cubic symmetric graph of order $2n$ of type $\{2^2\}$. 
Then $X$ has no odd automorphisms. 
\end{proposition}

\begin{proof}
First, we note that, by Proposition~\ref{pro:girth}, 
\begin{equation}
\label{girth}
\textrm{$X$ is of girth greater than $9$.}
\end{equation}
Suppose first that $X$ is not bipartite.
Then it must have only even automorphisms for otherwise
there would be an intransitive index $2$ subgroup of even automorphisms
in $\Aut(X)$ given rise to a bipartition of $X$. We may therefore
assume that $X$ is bipartite.
By (\ref{G22}) the automorphism group $\Aut(X)$ is
generated by an automorphism $h$ of order $3$ and an edge reversing 
automorphism $a$ of order $4$. Since $h$ is of odd order it is an even permutation.
To complete the proof we need to show that $a$ is an even permutation too. 

Consider the orbits of $\la a\ra$: they are either of length $2$ or  $4$.
Orbits of $\la a\ra$ of length $2$ are of two types: 
those with an inner edge (we refer to them as orbits of {\em type 1}) 
and those with no inner edge (we refer to them as orbits of {\em type 2}).
Finally, orbits of $\la a\ra$ of length $4$  will be referred to as orbits of {\em type 3}.
Let us remark that  orbits of type 3 have no inner edges because such an edge
would either give rise to a $4$-cycle in $X$ 
(contradicting (\ref{girth})) 
or would give rise to an edge flipped by the
involution $a^2$, which is not possible in 
cubic symmetric graphs of type $\{2^2\}$.
Observe that the numbers of orbits of type 1 and type 2 
are of the same parity when the bipartition sets are of even cardinality 
(that is, for $n$ even), and are of different parity when the bipartition sets are of 
odd cardinality (that is, for $n$ odd).

Let us now consider the quotient graph $X_a$ of $X$ 
with respect to the set of orbits of $\la a\ra$.
A  type 1 orbit  in $X_a$ cannot be adjacent to another type 1 orbit,
for this would give a cycle of length $4$ in $X$, contradicting (\ref{girth}).
Further, a type 1 orbit cannot be adjacent to two type 2 orbits,
for this would imply that the involution $a^2$ has a rigid cell containing the $H$-tree
as a subgraph, which, by Proposition~\ref{pro:cell212}, 
is impossible in a $2$-regular graph.
The only remaining possibility therefore is that an orbit of type 1 is adjacent to
an orbit of type 3. In other words, 
each orbit of type 1 is paired off with an
orbit of type 3.

We now turn to orbits of type 2. Such an orbit cannot be adjacent to three 
other orbits of type 2, for otherwise the involution $a^2$ would have a 
$Y$-tree as a rigid cell, contradicting Proposition~\ref{pro:cell212}.
The only remaining possibility is that an orbit of type 2
is adjacent to an orbit $O$ of type 3 and another orbit of type 2, the latter
being adjacent also to an orbit of type 3 which has to be different from $O$ 
in view of (\ref{girth}).
In conclusion, orbits of type 2 come in pairs, implying that 
\begin{equation}
\label{eq:orbits}
\textrm{the number of orbits of type 2 is even.}
\end{equation}
Moreover, each such pair of orbits of type 2 has two
neighbors in $X_a$, namely a pair of orbits of type 3.

We now consider orbits of type 3.
Recall that they have no inner edges.
Further,  by (\ref{girth}) two adjacent orbits of type 3 must be
 joined by a single matching.
We conclude that in $X_a$ an orbit of
type 3 is either adjacent to three orbits of type 3
 or to two orbits of type 3 and an orbit of length $2$
(either  of type 1 or of type 2). 
It follows that there are only five possible local adjacency structures of 
orbits   in $X_a$, those shown in Figure~\ref{fig:22}.

\begin{figure}[h!]
\begin{center}
\includegraphics[width=90mm]{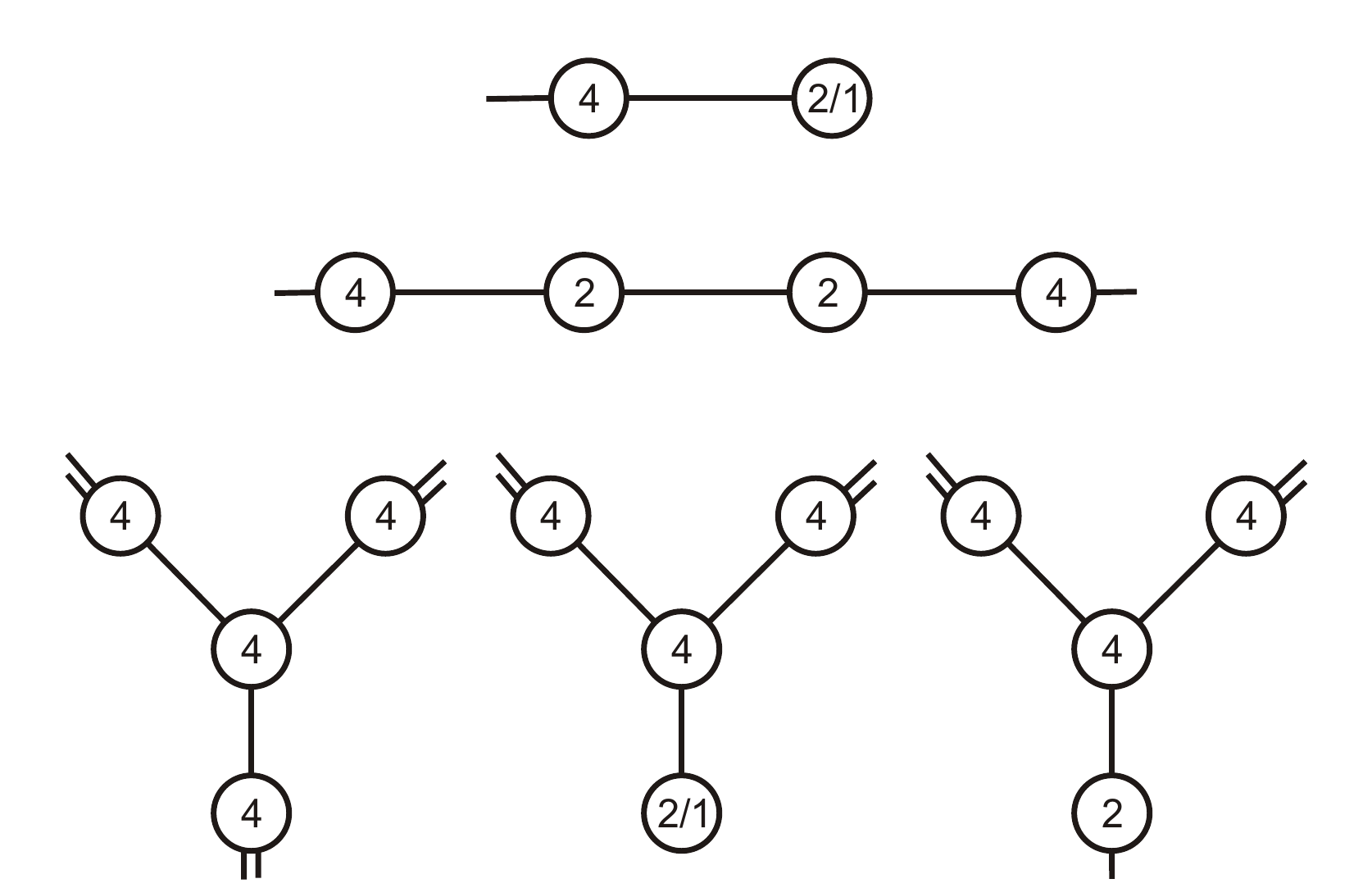}
\caption{\label{fig:22} {\footnotesize Possible local structures of orbits in $X_a$. }}
\end{center}
\end{figure}

The rest of the argument depends on the parity of $n$.
Suppose first that $n$ is even.
Since in this case orbits of type 1 and type 2 are of the same parity,
the number of orbits of type 1 is even by (\ref{eq:orbits}). 
Consequently, the number of orbits of type 3 adjacent to some orbit of length $2$
is therefore also even. 
To see that the total number of all orbits of $\la a\ra$ is even,
it suffices to show that
the number of orbits of type 3 which are only adjacent
to orbits of type 3 is even, too. 
We do this by modifying the quotient graph $X_a$
with the removal of all orbits of length $2$, and edges incident with them.
We obtain a graph which is a subdivision of a cubic graph.
Such a graph must have an even number of vertices of valency $3$ 
meaning that the number of orbits of type 3 not adjacent to orbits
of length $2$ was even to start with.
In summary, the number of orbits for each of the three types is even.
It follows that $a$ is an even automorphism. 

Suppose now that $n$ is odd.  
Then there is an odd number of orbits of type 1,
and hence also an odd number of orbits of type 3
adjacent to these orbits. Recall that by (\ref{eq:orbits})
we have an even number of orbits
of type 2, and consequently an even number of orbits of type 3 as their neighbors.
As in the case $n$ even, we remove all orbits of length $2$ 
(an odd number of orbits) from $X_a$. 
We obtain a graph which is a subdivision of a cubic graph in which
there is an odd number of vertices of valency $2$ and an even
number of vertices of valency $3$. Therefore $\la a\ra $
has an odd number of orbits of length $2$ and an odd number of orbits of
length $4$, in total an even number of orbits. This implies that $a$ is
an even automorphism, completing the proof of Proposition~\ref{pro:cubic-2-2}.
\end{proof}

The next proposition shows that odd automorphisms exist
in cubic symmetric graphs of type $\{4^2\}$ if and only
if the order of the graph is congruent to $2$ modulo $4$.

\begin{proposition}
\label{pro:cubic-4-2x}  
Let $X$ be a cubic symmetric graph of order $2n$
and of type $\{4^2\}$.
Then $X$ has odd automorphisms if and only if $n$ is odd. 
%{\color{red} Ali obstajajo taki grafi? Najmanjsi znan primer reda $3^10 x 468$.}
\end{proposition}

\begin{proof}
As in the proof of Proposition~\ref{pro:cubic-2-2} we note that, 
by Proposition~\ref{pro:girth}, 
\begin{equation}
\label{girth1}
\textrm{$X$ is of girth greater than $9$.}
\end{equation}
For the same reasons as in the proof of Proposition~\ref{pro:cubic-2-2}
we may assume that $X$ is bipartite.
By (\ref{G42}) $\Aut(X)$ is generated by an element $h$ of order $3$,
an edge-flipping automorphism $a$ of order $4$,
and three conjugate involutions $p$, $q$, and $r$
(of which $p$ is a square of $a$).
Clearly  Proposition~\ref{pro:cubic-4-2x} will be proved provided we 
show that $a$ is an odd automorphism if and only if $n$ is odd.
For this purpose we now count the orbits of $\la a\ra$.
In doing so we use an approach analogous to that used in the proof of 
Proposition~\ref{pro:cubic-2-2}.
 
The orbits of $\la a\ra$ are of three types:
orbits of length $2$ with an inner edge (type 1),
orbits of length $2$ without inner edges (type 2),
and orbits of length $4$ (type 3). The latter have no inner edges as an edge
would either give rise to a $4$-cycle in $X$ 
(which is impossible by (\ref{girth1})) 
or would give rise to an edge flipped by the
involution $a^2$, contradicting non-existence of edge-flipping involutions in 
cubic symmetric graphs of type $\{4^2\}$.
As in the proof of Proposition~\ref{pro:cubic-2-2} we 
observe that the numbers of orbits of type 1 and type 2 
are of the same parity when the bipartition sets are of even cardinality 
(that is, for $n$ even), and are of different parity when the bipartition sets are of 
odd cardinality (that is, for $n$ odd).

Let us now analyze the quotient graph $X_a$ of $X$ 
with respect to the set of orbits of $\la a\ra$ and the possible adjacencies between the
orbits of the three types. First we show that an orbit of type 1 must be adjacent to
two orbits of type 2. It certainly cannot be adjacent to an orbit of type 3, for
otherwise the involution $a^2$ would have an $I$-tree as a rigid cell,
contradicting Proposition~\ref{pro:cell24}.
Further, an orbit of type 1 cannot be adjacent to another orbit of type 1,
for this would give rise to a cycle of length $4$ in $X$, contradicting (\ref{girth1}).
Hence, an orbit of type 1 is adjacent to two orbits of type 2.
The remaining two neighbors in $X_a$
of these two orbits of type 2 must be orbits of type 3. 
Namely, if one of these remaining neighbors
was an orbit of type 2 then $a^2$ would fix a 
path of length $4$ in $X$, which is clearly not possible in a $4$-regular graph. 

Consider now orbits of type 2.
If such an orbit is adjacent to three orbits of type $2$ and 
if each of these three orbits has an orbit of type 3 as a neighbor
then $a^2$ has the $Y$-tree for a rigid cell, 
contradicting Proposition~\ref{pro:cell24}. 
Therefore at least one of these 
three orbits is adjacent to two additional orbits of length $2$
(which must be of type 2, of course).
In fact, precisely one of these three orbits has such neighbors, as otherwise  
$a^2$ would fix a path of length $4$ in $X$, contradicting $4$-regularity of $X$. 
These six orbits of type 2 give rise to an $H$-tree in $X_a$.
The remaining neighbors of the four vertices of valency $1$ in this $H$-tree
must all be orbits of type 3, for otherwise $a^2$ would fix a path of length 
greater than or equal to $4$.
It follows that there are only two possible local adjacency structures of 
orbits of length $2$ in $X_a$ -- those shown in Figure~\ref{fig:42}.
Consequently, the number of adjacencies between orbits of type 3
with orbits of type 2 is even (and there are no adjacencies between orbits of type 1
and orbits of type 3).

\begin{figure}[h!]
\begin{center}
\includegraphics[width=80mm]{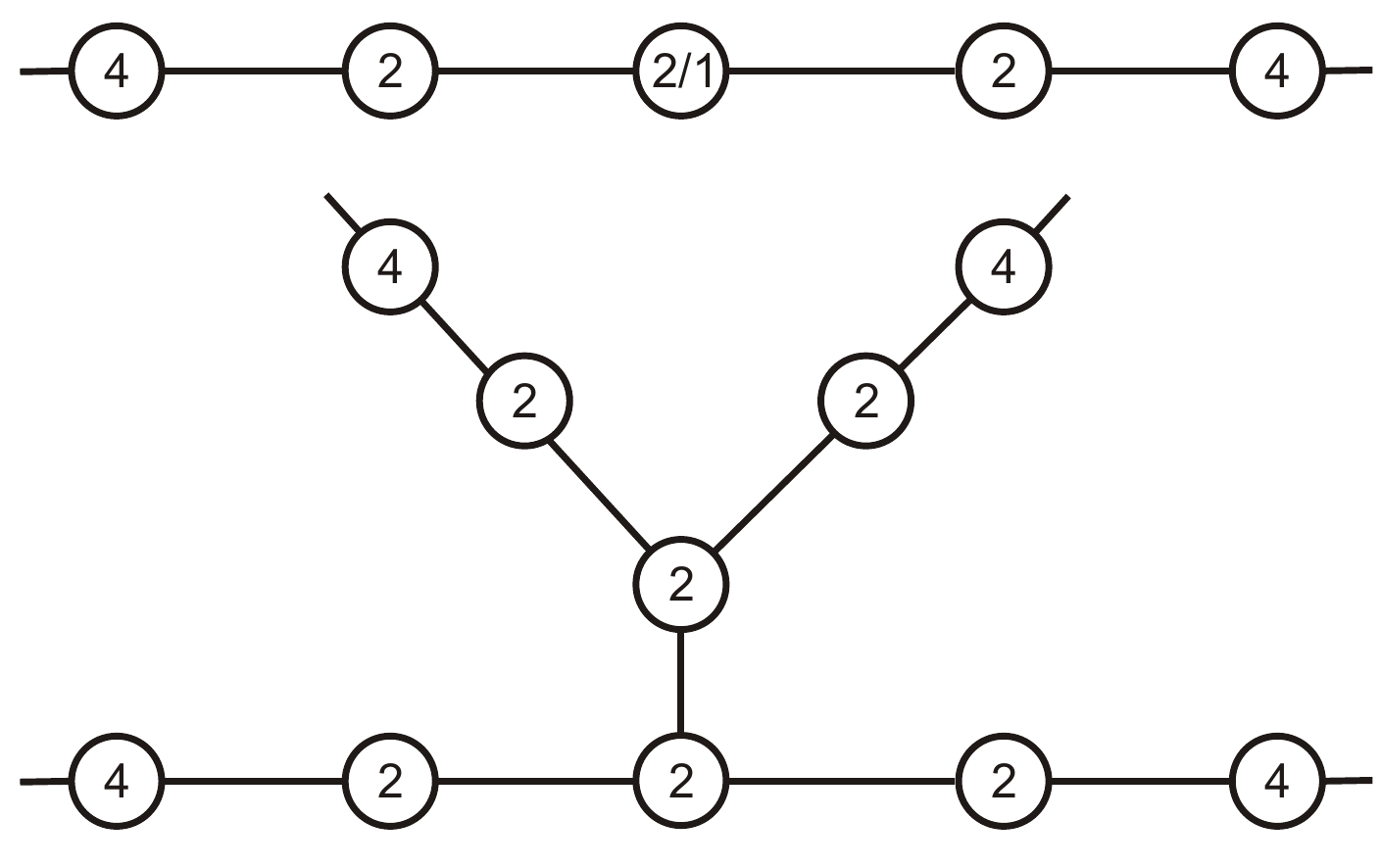}
\caption{\label{fig:42} {\footnotesize Possible local structures of orbits of length $2$ in $X_a$.}}
\end{center}
\end{figure}

%Note that a vertex in an orbit of $\la a\ra$ of length $4$ cannot have
%two neighbors in an adjacent orbit of length $4$ (namely $X$ is of girth $>9$).

Consider now orbits of type 3. 
Recall that there are no edges inside an orbit of type 3.
Further,  by (\ref{girth1}), two adjacent orbits of type 3 must be
 joined by a single matching.
We can conclude that an orbit of
type 3 is either adjacent to three orbits of type 3
in $X_a$ or to two orbits of type 3 and one orbit of type 2. 

We will now modify the quotient graph $X_a$ by  removing all orbits 
of length $2$, obtaining thus a subdivision of a cubic graph. 
Of course the number of vertices of valency $3$ in this graph 
is even. Also,
we know that the number of vertices of valency $2$ (corresponding to
orbits of type 3 with one neighbor being an orbit of type 2) is even.
Consequently, the number of orbits of type 3 is even.
Since we know that the number of orbits of length $2$
is even if and only if $n$ is even, it follows that it is the parity of $n$ that determines
whether the number of orbits of $a$ is even or odd.
Consequently,
$a$ is an odd permutation if and only if $n$ is odd.
\end{proof}

\begin{remark}\label{remark}{\rm
There are no graphs of type $\{4^2\}$ in Conder's list of all cubic 
symmetric graphs up to order $10.000$ \cite{CD}. However,
such graphs do exist \cite{CN07}. Known examples in the literature
are of order ${0\pmod 4}$. Marston D.~E. Conder
has kindly pointed out to us that there exists
a graph of order $5314410 = 2^{1}\cdot 3^{12}\cdot 5^{1}\equiv {2\pmod 4}$.
In particular, the group $G_4^2$ has a quotient $Q$ of order 
$2^{4}\cdot 3^{17}\cdot 5^{1}$
that is an extension of an elementary abelian normal $3$-subgroup of order $3^{15}$ 
by $M_{10}$ (the point stabilizer in the Mathieu group $M_{11}$).  The group $Q$ 
act $4$-regularly on a cubic $5$-arc-transitive graph $Y$ (of order $430467210$), 
but it has two normal subgroups of order $3^4$ that are interchanged under 
conjugation by an element of the full automorphism group of $Y$ (not lying in $Q$). 
Factoring out one of these normal subgroups gives a quotient $P$ of order 
$2^{4}\cdot 3^{13}\cdot 5^{1} = 127545840$, 
which  is the full automorphism group of a 
cubic $4$-regular graph  of type $\{4^2\}$ of order 
$127545840/24 = 5314410$. 
This graph is a regular $\ZZ_3^{11}$-cover of 
Tutte's $8$-cage $\textrm{F}030\textrm{A}$.
By Proposition~\ref{pro:cubic-4-2x} 
it admits odd automorphisms.} 
\end{remark}

The following proposition considers all remaining types of cubic $s$-regular graphs,
$s\ge 2$, without $(s-1)$-regular subgroups in their automorphism groups.
It shows, for example, that amongst the graphs $\mathrm{F}110\mathrm{A}$, $\mathrm{F}182\mathrm{D}$ 
and $\mathrm{F}506\mathrm{A}$,
the three smallest examples of cubic symmetric graphs of type $\{3\}$, only 
the graph $\mathrm{F}506\mathrm{A}$ has no odd automorphisms. Namely, all these three graphs are of order
${2\pmod 4}$, but $\mathrm{F}506\mathrm{A}$ is not bipartite whereas the other two graphs are bipartite.
 
\begin{proposition}
\label{pro:cubic-s-regular-without-s-1}  
Let $X$ be a cubic symmetric graph of order $2n$ 
admitting an $s$-regular subgroup $G\le\Aut(X)$,
which is an epimorphic image of a group different from  
$G_2^2$ and $G_4^2$, 
with no $(s-1)$-regular subgroup, where $s\ge 2$.
Then there exists an odd automorphism of $X$ in $G$  if and only if 
$n$ is odd and $X$ is bipartite.
% or $G$ is  an epimorphic image of  $G_4^2$. 
\end{proposition}

\begin{proof}
Assume that there exists an odd automorphism of $X$ in $G$.
Then there exists a subgroup $K\le G$ of index $2$ consisting of even automorphisms.
Since an element of order $3$ from $G_v$ lies in $K$ and there is no $(s-1)$-regular
subgroup of $G$ we can conclude that $K$ is intransitive with two orbits on $V(X)$,
forcing $X$ to be bipartite.
Moreover, odd automorphisms in $G$ are precisely 
those automorphisms that interchange the bipartition sets $X$.
Since $G$ is not an epimorphic image of either $G_2^2$ or  $G_4^2$ there
exists an involution $a\in G\setminus K$ flipping an edge in $X$. 
This involution $a$ is clearly semiregular, and thus it 
is odd if and only if $n$ is odd.

Conversely, if $X$ is bipartite and $n$ is odd then
we can reverse the argument of the paragraph above 
to conclude that the involution $a$ flipping an edge is an odd automorphism.
\end{proof}

In what follows,
Propositions~\ref{pro:cubic-2-regular-2mod4} 
and~\ref{pro:cubic-2-regular-0mod4} combined together lay the ground 
for the proof of Theorem~\ref{thm:main} 
for cubic symmetric graphs of type $\{1,2^1\}$,
while Proposition~\ref{pro:Janko} recalls a group-theoretic result
needed in the proof of 
Proposition~\ref{pro:cubic-2-regular-0mod4}.

\begin{proposition}  
\label{pro:cubic-2-regular-2mod4}
Let $X$ be a cubic symmetric graph of order $2n$, where $n$ is odd, and 
of type $\{1,2^1\}$.
Then $X$ has odd automorphisms. 
\end{proposition}

\begin{proof}
Since $X$ is of type $\{1,2^1\}$ there exists  a $1$-regular subgroup $K$ 
of $\Aut(X)$. By Proposition~\ref{pro:cubic-1-regular}, $K$ contains 
odd automorphisms, and the result follows.
\end{proof}

Given a group $G$ and a subset $S \subseteq G$ we let $N_G(S)$ and $C_G(S)$ 
be the normalizer and the centralizer of $S$ in $G$, respectively.
Recall that in a transitive permutation group $G$ acting on a set $V$, 
the number of points left fixed by $G_v\le G$, $v\in V$, 
is equal to the index $|N_G(G_v)\colon G_v|$
of $G_v$ in its normalizer $N_G(G_v)$ (see, for example, \cite[Theorem~3.5]{W67}).

Proposition~\ref{pro:cubic-2-regular-0mod4}
gives the answer to the question about existence of odd automorphisms
in graphs of order $0\pmod{4}$ which admit an epimorphic image of $G_2^1$
as a $2$-regular subgroup in their automorphism groups.
In the proof of this proposition we will use the following Janko's result 
(which can be extracted from \cite[Theorem~1.1]{J01})
about finite $2$-groups having a small centralizer of an involution.  

\begin{proposition}
\label{pro:Janko}
Let $G$ be a finite non-abelian $2$-group containing an involution $t$ such that the centralizer $C_G(t)$ is of the form $\la t\ra \times C$,
where $C$ is a cyclic group of order $2$. Then either $G$ is a dihedral group of order $2^n$, $n\ge 3$, or $G$
is a quasi-dihedral group $QD_{2^n}$ of order $2^n$, $n\ge 4$.
\end{proposition}

Given integers $m\ge 1$ and $n\ge 2$, a group of automorphisms of a
graph is called $(m,n)$-{\em semiregular} if it has $m$ orbits of
length $n$ and no other orbit, and the action is regular on each
orbit. An $m$-{\em Cayley graph} $X$ {\em on a group} $H$ is a graph admitting an
$(m,n)$-semiregular subgroup of automorphisms isomorphic to $H$. 

\begin{proposition}  
\label{pro:cubic-2-regular-0mod4}
Let $X$ be a cubic symmetric graph of order $2^km$, where $k\ge 2$ 
and $m\ge 1$ is odd (thus, $X$ is of order ${0\pmod 4}$), 
with a $2$-regular subgroup $G\le \Aut(X)$, which 
is an epimorphic image of $G_2^1$.
Then $G$ contains an odd automorphism if and only if $X$ is 
non-bipartite, of type $\{1,2^1\}$ an $m$-Cayley graph on a cyclic group of order 
$2^k$.
\end{proposition}

\begin{proof}
Observe that with the assumption on $k$ and $m$, clearly an 
$m$-Cayley graph on a cyclic group of order $2^k$ admits odd automorphisms:
such an automorphism is a generator of the cyclic group of order $2^k$.

To complete the proof of Proposition~\ref{pro:cubic-2-regular-0mod4}
assume that $X$ has an odd automorphism in $G$. 
Since $G$ is $2$-regular we have $|G|=3\cdot 2^{k+1} m$.
By (\ref{G21}), $G$ is generated by an element of order $3$, 
which cannot be an odd automorphism, and by two involutions $a$ and $p$. 
Therefore $X$ has an odd automorphism in $G$ if and only if 
at least one of these two involutions 
is an odd automorphism.
Since $X$ is of order ${0\pmod 4}$, an involution is odd if and only if it fixes an odd 
number of pairs of adjacent vertices.
Hence, let  $t\in G$ be an odd involution. Then 
$|N_G(\la t\ra)\colon \la t\ra|=|C_G(t)\colon \la t\ra|=2m'$, where $m'\ge 1$ is odd,
implying that $|C_G(t)|=4m'$. 
Let $P$ be a Sylow $2$-subgroup of $G$ containing $t$.
Note that $|P|=2^{k+1}\ge 8$.
Since $t\in C_P(t)\le C_G(t)$ is an involution it cannot be odd inside
a cyclic subgroup of order $4$. It follows that 
$C_P(t)\cong\la t\ra\times \ZZ_2$. 
If $P$ is abelian then $P\le C_G(t)$, implying that $P=C_P(t)\cong\la t\ra\times \ZZ_2$, 
and thus $|P|=4$, contradicting the assumption
that $|P|=2^{k+1}\ge 8$.
Therefore, $P$ is non-abelian, and so Proposition~\ref{pro:Janko} implies that $P$ 
contains a cyclic subgroup $C$ of order $2^{k}$.
Since $C_P(t)$ is isomorphic to $\ZZ_2\times\ZZ_2$ 
it follows that $t\in P\setminus C$.
In particular, since $t\in P\cap G_v$, $v\in V(X)$ and $|G_v|=6$, 
a generator of $C$ is semiregular on $V(X)$.
It follows that $X$ is an $m$-Cayley graph on $C$. 
 
To complete the proof we need to show that 
the existence of odd automorphisms in $X$ implies that
$X$ is a non-bipartite graph of type $\{1,2^1\}$.
Suppose on the contrary that $X$ is bipartite. Then
the odd involution $t\in G$, which is contained in some vertex stabilizer, 
fixes the two bipartition sets of $X$. Therefore any orbit %$\cal{O}$
of $\la t \ra$  is contained in one of the two bipartition sets.
Further, $\textrm{Fix}(t)$ splits equally into the two bipartition sets.
Namely, every vertex fixed by $t$ in one bipartition set has a neighbor in the
other bipartition set which is also fixed by $t$.
In view of the fact that the two bipartition sets are of equal cardinality,
this implies that the two bipartition sets contain the same number of orbits of $\la t\ra$
of length $2$ also, implying that the total number of orbits of length $2$ is even,
contradicting the fact that $t$ is an odd automorphism. 
Therefore $X$ is non-bipartite. Finally,   
Proposition~\ref{pro:cubic-s-regular-without-s-1} and the fact that graphs  
of types $\{1,2^1,2^2,3\}$ and $\{2^1,2^2,3\}$ are all bipartite 
(see Table~\ref{tab:cubic-types}) combined together
imply  that $X$ is of type $\{1,2^1\}$.
\end{proof}

The next proposition deals with graphs  of 
types $\{2^1,3\}$, $\{2^2,3\}$, $\{2^1, 2^2,3\}$
and $\{1, 2^1, 2^2,3\}$.

\begin{proposition}
\label{pro:3-2-odd}  
Let $X$ be a cubic symmetric $3$-regular graph of order $2n$
admitting  a $2$-regular subgroup $K$ 
in its full automorphism group $\Aut(X)$.  
Then $X$ has odd automorphisms if and only if $n$ is odd. 
\end{proposition}

\begin{proof}
By Proposition~\ref{pro:cell23} there exists 
an involution $\sigma\in \Aut(X)\setminus K$ 
with all of its rigid cells isomorphic to the $Y$-tree. 
Therefore the number of fixed vertices of $\sigma$ is divisible by $4$.
Consequently, $\sigma$ is an odd permutation if and only if $n$ is odd.

To complete the proof let us assume that $n$ is even, and suppose, by contradiction,
that $X$ contains odd automorphisms. 
Suppose first that $X$ contains an odd automorphism in $\Aut(X)_v$, $v\in V(X)$.
This implies that $\Aut(X)_v\setminus K_v$ 
consists of odd automorphisms and that $K_v$
consists of even automorphisms. 
But this contradicts the above statement about involutions in vertex stabilizers
being odd if and only if $n$ is odd.  Consequently, 
odd automorphisms of $X$ can possibly only exist outside vertex
stabilizers.  Let $H$ be the index $2$ subgroup of $\Aut(X)$ consisting
of even automorphisms only.
But note that the whole of $\Aut(X)_v$ is contained in $H$, 
and so $H$ is an index $2$ intransitive subgroup of $\Aut(X)$, 
forcing  $X$ to be bipartite.
Furthermore, all automorphisms interchanging the two bipartition sets 
are odd permutations.
By (\ref{G3}), there exists an involution $a\in \Aut(X)$ flipping an edge of $X$
and thus interchanging the bipartition sets. It follows that $a$ has
$n$ orbits of length $2$ (and no fixed vertices), and so is an even permutation,
a contradiction. 
\end{proof}

The last proposition in the series solves the problem of existence of
odd automorphisms in cubic symmetric graphs of types $\{4^1,5\}$, $\{4^2,5\}$,
$\{4^1, 4^2,5\}$ and $\{1,4^1, 4^2,5\}$. Since cubic symmetric graphs
of types $\{4^1,5\}$ and $\{4^2,5\}$ are never bipartite 
(see Table~\ref{tab:cubic-types}),
Proposition~\ref{pro:5-4-oddx} implies that none of these graphs admits odd 
automorphisms. 
An example of a graph of type $\{4^2,5\}$ is the graph
$\mathrm{F}234\mathrm{B}$,
the only cubic $5$-regular vertex-primitive graph \cite{W67}. 
On the other hand, since cubic symmetric graphs
of types $\{4^1, 4^2,5\}$ and $\{1,4^1, 4^2,5\}$ are always bipartite (see 
Table~\ref{tab:cubic-types}) Proposition~\ref{pro:5-4-oddx} implies
that graphs of these  two types admit  odd 
automorphisms if and only if they are of order twice an odd number. 

\begin{proposition}
\label{pro:5-4-oddx} 
Let $X$ be a cubic symmetric $5$-regular graph of order $2n$
admitting  a $4$-regular subgroup $K$ 
in its full automorphism group $\Aut(X)$.  
Then $X$ has odd automorphisms if and only if 
$X$ is bipartite and $n$ is odd.  
\end{proposition}

\begin{proof}
Suppose first that 
$X$ is bipartite and $n$ is odd. Then
the edge-flipping involution $a\in \Aut(X)$, 
which exists by (\ref{G5}), is semiregular with $n$ orbits of length $2$,
and thus an odd automorphism.

In the proof of the converse statement the following two claims will be needed.
 
\medskip

\noindent
%\begin{claim}
{\sc Claim~1.} {\em There exists a $4$-regular subgroup $M$ of $\Aut(X)$ 
%(either  an epimorphic image $G_4^1$ or of  $G_4^2$) 
such that every element of $M_v$, $v\in V(X)$, is an even automorphism. }
%\end{claim}

\medskip

\noindent 
Since $X$ contains odd automorphisms
there exists an index $2$ subgroup of even automorphisms,
which is either 
arc-transitive and hence an epimorphic image of  $G_4^1$ or $G_4^2$
(in which case the claim is clearly true)
or it is intransitive. If the latter occurs then the whole of $\Aut(X)_v$, $v\in V(X)$, 
consists of even automorphisms, and hence also the whole of $M_v$ 
consists of even automorphisms. 
This proves the above claim. 

\medskip

\noindent
%\begin{claim}
{\sc Claim~2.} {\em The stabilizer $\Aut(X)_v$, $v\in V(X)$,
consists of even automorphisms. }
%\end{claim}

\medskip
 
\noindent
By Proposition~\ref{pro:cell25}, 
for every vertex $v\in V(X)$ there exists 
a unique canonical 
involution $\sigma=\sigma_v$ for $v$
(with $A(v)$ as its rigid cell).
Note also that, by \cite[Lemma~1]{D79}, 
$\sigma_v$ is a central element of $\Aut(X)_v\cong S_4\times \ZZ_2$.
Let $v=w_0$, $w_1$, \ldots, $w_{m-1}$ be those vertices of $X$ for which 
$A(w_i)$, $i\in\ZZ_m$, are all $\sigma$-rigid cells.

Consider the orbits of $\la\sigma\ra$.
Orbits of $\la \sigma\ra$ are of length $1$ and $2$.
Exactly $10m$ orbits of $\la \sigma\ra$ are of length $1$
(each $\sigma$-rigid cell gives $10$ such orbits).
Orbits of $\la \sigma\ra$ of length $2$ are of two types: 
those with an inner edge (we refer to them as orbits of {\em type 1}) 
and those with no inner edge (we refer to them as orbits of {\em type 2}).
Let $j$ be the number of orbits of $\la\sigma\ra$ of type 1 and 
$j'$ be the number of orbits of $\la\sigma\ra$ of type 2.
Then 
\begin{eqnarray}
\label{eq:j}
|V(X)|&=&2j'+2j+10m = 2n. 
\end{eqnarray}
Let us now consider the quotient graph $X_\sigma$ of $X$ 
with respect to the set of orbits of $\la \sigma\ra$.
Let us modify  the quotient graph $X_\sigma$
with the removal of all orbits of length $1$, and orbits of type 1 adjacent
to these orbits (note that there are $6m$ such orbits 
as orbits of length $1$ form an $A$-tree).  
We obtain a graph which is a subdivision of a cubic graph.
Such a graph must have an even number of vertices of valency $3$, 
meaning that $j'-6m$, and thus also $j'$, is an even number.

Suppose first that $j=0$. Then, by (\ref{eq:j}), $m$ and $n$ are of the same parity, implying that  
$\sigma$ has an even number of orbits
of length $2$, and is thus an even automorphism.
Since, by (\ref{G5}), all involutions $p=\sigma$, $q$, $r$ and $s$ 
in a generating set of $\Aut(X)$ are conjugate, we may conclude that $\Aut(X)_v$, $v\in V(X)$,
consists of even automorphisms as claimed.
We may therefore assume that $j\ne 0$. 
(Note that in this case $X$ is non-bipartite in view of the fact that there exist
both a vertex fixed by $\sigma$ and an edge flipped by $\sigma$. 
By Table~\ref{tab:cubic-types}, $X$ is either of type $\{4^2,5\}$ or of type  $\{4^1,5\}$.)
By (\ref{eq:j}) it suffices to show that $j$ is even.

Let $\mathcal{R}$ be the set of $j$ edges flipped by $\sigma$.
Since $\sigma$ is the central element in $\Aut(X)_v\cong S_4\times\ZZ_2$, one can easily see that
$\mathcal{R}^\gamma=\mathcal{R}$ for each $\gamma \in \Aut(X)_v$.
Since $|\Aut(X)_v|=48=3\cdot 2^4$ the length of every orbit of $\Aut(X)_v$ on $\mathcal{R}$ is a divisor of $48$.
Further, since the stabilizer of an edge in $X$ is a $2$-group, the element of order $3$ in $\Aut(X)_v$
cannot fix an edge in $\mathcal{R}$, and so the length of every orbit of $\Aut(X)_v$ on $\mathcal{R}$ is divisible by $3$.
In fact the lengths of orbits of $\Aut(X)_v$ on $\mathcal{R}$ are divisible by $6$.
Suppose on the contrary that there exists an orbit of $\Aut(X)_v$ on $\mathcal{R}$ of length $3$.
Then the intersection $J$ of the edge stabilizer $\Aut(X)_e$, $e=uu'$ with $\Aut(X)_v$
is a Sylow $2$-group of $\Aut(X)_v$, and thus
isomorphic to $D_8\times \ZZ_2$. Since $\sigma$ flips the edge $e=uu'$, it follows that half of
the elements of $J$ fix $\{u,u'\}$ point-wise whereas the other half  
interchange $u$ and $u'$. Note however that the edge-stabilizer in a $5$-regular cubic graph is isomorphic to
$(D_8\times\ZZ_2)\rtimes\ZZ_2\cong\la z,x,y\mid z^8=x^2=y^2=1,xzx=z^{-1}, yzy=z^5\ra$ and 
has $7$ subgroups of order $16$. Amongst these $7$ subgroups
only the stabilizer of the arc $(u,u')$
 is isomorphic to $D_8\times \ZZ_2$,
 a contradiction.
It follows that the length
of every orbit of $\Aut(X)_v$ on $\mathcal{R}$
is divisible by $6$. We conclude that $\mathcal{R}$ is of even cardinality,
that is, $j$ is even.
It follows by (\ref{eq:j}) that $m$ and $n$ are of the same parity, and consequently  
  $\sigma$ has an even number of orbits
of length $2$, and is therefore an even automorphism.
Since, by (\ref{G5}), all involutions $p=\sigma$, $q$, $r$ and $s$ 
in a generating set of $\Aut(X)$ are conjugate, we may conclude that $\Aut(X)_v$, $v\in V(X)$,
consists of even automorphisms as claimed.
This completes the proof of Claim~2.

\medskip

We are now ready to prove that the existence of odd automorphisms in $X$
implies that $X$ is bipartite of order ${2\pmod 4}$.
So let us assume that there exists an odd automorphism in $\Aut(X)$.
In view of Claim~2 vertex stabilizers 
consist of even automorphisms, and so
$\Aut(X)$ contains an intransitive subgroup of index $2$  consisting of even 
automorphisms, forcing $X$ to be  bipartite. 
In addition, all automorphisms interchanging the two bipartition sets are odd.
In particular, the edge-flipping involution $a\in \Aut(X)$, 
which exists by (\ref{G5}), is odd.  It follows that the bipartition 
sets are of odd cardinality, that is, $n$ is odd. 
This completes the proof of Proposition~\ref{pro:5-4-oddx}.
\end{proof}

\medskip

The results contained in 
Propositions~\ref{pro:cubic-1-regular} -- \ref{pro:cubic-2-regular-2mod4} 
and~\ref{pro:cubic-2-regular-0mod4} -- \ref{pro:5-4-oddx}
now establish the proof of Theorem~\ref{thm:main}.

\medskip

\begin{proofT}
Let $X$ be a cubic symmetric graph. Then $X$ is of one of   $17$ possible 
types (see Subsection~\ref{sec:cubic-sym}).
Part (i) of Theorem~\ref{thm:main} follows from 
Proposition~\ref{pro:cubic-1-regular} for graphs of type $\{1\}$,
from 
Proposition~\ref{pro:cubic-4-2x} for 
graphs of type $\{4^2\}$, from Proposition~\ref{pro:3-2-odd} 
for graphs  
of types $\{1,2^1,2^2,3\}$, $\{2^1,2^2,3\}$, $\{2^1,3\}$ and $\{2^2,3\}$,
and from Proposition~\ref{pro:5-4-oddx} for graphs of types
$\{1,4^1,4^2,5\}$ and $\{4^1,4^2,5\}$.
Since every cubic symmetric graph of type $\{1,4^1\}$ is bipartite
(see Table~\ref{tab:cubic-types}), part (i) of Theorem~\ref{thm:main}  
for graphs of type $\{1,4^1\}$
follows from Proposition~\ref{pro:cubic-s-regular-without-s-1}.

Part (ii) of Theorem~\ref{thm:main} follows from 
Proposition~\ref{pro:cubic-s-regular-without-s-1}.

Part (iii) of Theorem~\ref{thm:main} follows from 
Propositions~\ref{pro:cubic-2-regular-2mod4} 
and~\ref{pro:cubic-2-regular-0mod4}.

Finally, part (iv) of Theorem~\ref{thm:main} follows from Proposition~\ref{pro:cubic-2-2} for 
graphs of type $\{2^2\}$, and from Proposition~\ref{pro:5-4-oddx} for graphs  
of types $\{4^1,5\}$ and  $\{4^2,5\}$
(since none of graphs of types $\{4^1,5\}$ and $\{4^2,5\}$ is bipartite).
\end{proofT}

\bigskip
\footnotesize
\noindent\textit{Acknowledgments.}
The authors wish to thank  Marston D.~E. Conder, Ademir Hujdurovi\'c and Aleksander Malni\v c
for helpful conversations about the material in this paper.

The work of Klavdija Kutnar  was supported in part by the Slovenian Research Agency 
(research program P1-0285 and research projects N1-0032, N1-0038, J1-6720, J1-6743, 
and J1-7051), in part by WoodWisdom-Net+, W$^3$B, and in part by NSFC project 11561021.
The work of Dragan Maru\v si\v c was supported in part by the Slovenian Research Agency (I0-0035, research program P1-0285 and research projects N1-0032, N1-0038, J1-5433,  J1-6720,
and J1-7051), and in part by H2020 Teaming InnoRenew CoE.

\end{document}